\documentclass[10pt]{article}
\usepackage[utf8]{inputenc}
\usepackage{amsmath, amssymb, amsthm}
\usepackage[margin=1.25in]{geometry}
\usepackage{verbatim}
\usepackage[dvips]{graphicx}
\usepackage{pstricks}
\usepackage{breqn}

\begin{document}

\title{Generalizations of Triangle Inequalities to Spherical and Hyperbolic Geometry}
\author{Karina Cho and Jacob Naranjo}
\date{}

\newcommand{\ssf}[1]{\sin\frac{#1}{2}}
\newcommand{\hsf}[1]{\sinh\frac{#1}{2}}
\newcommand{\Jbar}{\bar J}

\newtheorem{theorem}{Theorem}[section]
\newtheorem{corollary}[theorem]{Corollary}
\newtheorem{lemma}[theorem]{Lemma}
\newtheorem{proposition}{Proposition}[section]
\newtheorem*{remark}{Remark}

\maketitle

\begin{abstract}
    Certain triangle inequalities involving the circumradius, inradius, and side lengths of a triangle are generalized to spherical and hyperbolic geometry. Examples include strengthenings of Euler's inequality, $R\geq2r$. An extension of Euler's inequality to a simplex in $n$-dimensional space is also generalized to spherical geometry.
\end{abstract}

\section{Introduction}

The Euclidean plane, characterized by constant Gaussian curvature $K=0$, gives rise to a multitude of triangle relations. For example, Euler's inequality relates the circumradius $R$ and inradius $r$ of a triangle as
\[ R\geq 2r. \]
We define spherical geometry by a surface of constant curvature $K=1$ and hyperbolic geometry by curvature $K=-1$. In \cite{svrtan_veljan}, Svrtan and Veljan generalize Euler's inequality as
\[
\begin{array}{cl}
R\geq 2r & \text{in Euclidean geometry}\\
\tan R\geq 2\tan r & \text{in spherical geometry}\\
\tanh R\geq 2\tanh r & \text{in hyperbolic geometry}
\end{array}
\]
maintaining the property that equality is achieved if and only if the triangle is equilateral. Our objective is to generalize other Euclidean relations in a similar manner.

Spherical geometry is modeled by the unit sphere in $\mathbb{R}^3$. For hyperbolic geometry we use the Klein disk model, where each point on the hyperbolic plane is identified with a point on the open unit disk. Unlike the Poincar\'e disk model, the Klein model is not conformal; instead, its advantage is that hyperbolic lines are straight lines. Given two points $p$ and $q$, we can find their hyperbolic distance $d(p,q)$ as follows: draw the line through $p$ and $q$ and label its points of intersection with the boundary of the disk as $a$ and $b$, so that the line contains $a$, $p$, $q$, and $b$ in that order. Then
\[
d(p,q)=\frac{1}{2}\ln\frac{|aq|\,|pb|}{|ap|\,|qb|}
\]
where vertical bars indicate Euclidean distances. This formula can be used to show that if a point is a Euclidean distance $r$ from the origin, its hyperbolic distance from the origin is
\[
d(r)=\frac{1}{2}\ln\frac{1+r}{1-r}=\tanh^{-1} r.
\]

\begin{figure}
	\begin{center}
		\includegraphics[width=0.85\columnwidth]{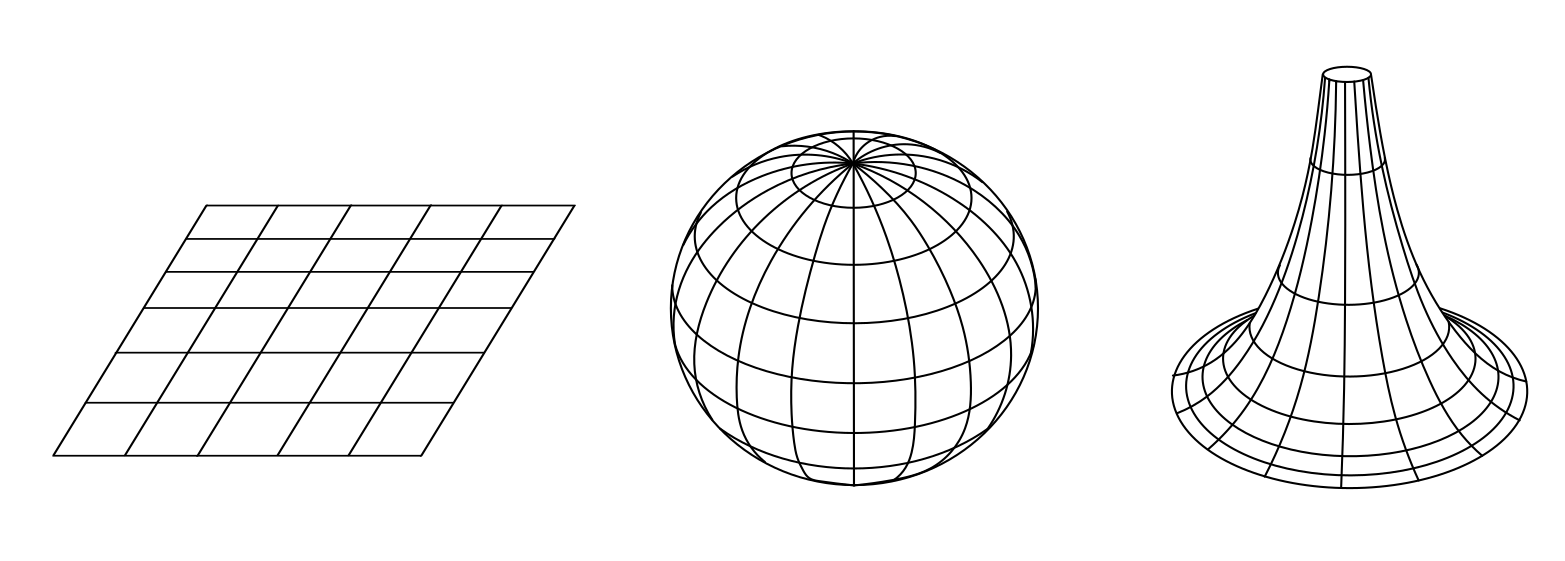}
		\caption{Surfaces of constant Gaussian curvature $K=0$, $K=1$, and $K=-1$, from left to right.}
	\end{center}
\end{figure}

Our generalizations will make use of the following unifying function, also used in \cite{reu2016} and \cite{guo}:
\begin{equation}\label{func:s}
	s(x):=\left\{
	\begin{array}{ll}
		\frac{x}{2} & \text{in Euclidean geometry}\\
		{\sin \frac{x}{2}} & \text{in spherical geometry}\\
		{\sinh \frac{x}{2}} & \text{in hyperbolic geometry}
	\end{array}
	\right.
\end{equation}
The following lemma illustrates a fundamental property of the $s$-function.

\begin{lemma}\label{lemma:exist}
	In Euclidean, spherical, or hyperbolic geometry, if there exists a triangle with side lengths $a$, $b$, and $c$, then there exists a Euclidean triangle with side lengths $s(a), s(b)$, and $s(c)$.
\end{lemma}

\begin{proof}
	For Euclidean geometry, the statement is trivial: if $a$, $b$, and $c$ satisfy the triangle inequality then the same is true for $a/2$, $b/2$, and $c/2$.
	
	For spherical geometry, consider a triangle with side lengths $a$, $b$, and $c$ on the unit sphere in $\mathbb{R}^3$. Without loss of generality we may assume all three vertices are at a spherical distance $R$ from the north pole $(0,0,1)$. By taking the projection of the triangle onto the tangent plane $z=1$ we obtain a Euclidean triangle with side lengths $2\ssf{a}/\cos R$, $2\ssf{b}/\cos R$, and $2\ssf{c}/\cos R$, or $k\cdot s(a)$, $k\cdot s(b)$, and $k\cdot s(c)$ with $k=2/\cos R$. Thus there exists a Euclidean triangle with side lengths $s(a)$, $s(b)$, and $s(c)$.
	
	For hyperbolic geometry, consider a triangle in the Klein disk model with side lengths $a$, $b$, and $c$. Without loss of generality we may assume all three vertices are at a hyperbolic distance $R$ from the origin. The Euclidean triangle determined by the same points has side lengths $k\cdot s(a)$, $k\cdot s(b)$, and $k\cdot s(c)$ with $k=2/\cosh R$. Thus there exists a Euclidean triangle with side lengths $s(a)$, $s(b)$, and $s(c)$.
\end{proof}

For triangle relations which involve only the side lengths, the $s$-function produces an immediate generalization to all three geometries.

\begin{lemma}[proved in \cite{reu2016}]
	\label{lemma:s}
	If $f(a,b,c)\geq 0$ holds for all triangles with side lengths $a,b,c$ in Euclidean geometry, then $f(s(a),s(b),s(c))\geq 0$ holds for all triangles with side lengths $a,b,c$ in all three geometries.
\end{lemma}

\begin{proof}
	Suppose a triangle in Euclidean, spherical, or hyperbolic geometry has side lengths $a,b,c$. By Lemma \ref{lemma:exist}, there exists a triangle in Euclidean geometry with side lengths $s(a),s(b),s(c)$. The lemma condition then implies $f(s(a),s(b),s(c))\geq 0$.
\end{proof}


The $s$-function also gives rise to unified formulas for the circumradius and inradius in terms of the side lengths $a,b,c$:
\begin{equation}\label{form:R}
	\frac{2s(a)s(b)s(c)}{\sqrt{s(a+b-c)s(a+c-b)s(b+c-a)s(a+b+c)}}= \left\{
	\begin{array}{ll}
		R & \quad\text{(Euclidean)}\\
		\tan R & \quad\text{(spherical)}\\
		\tanh R & \quad\text{(hyperbolic)}
	\end{array}
	\right.
\end{equation}
\begin{equation}\label{form:r}
	\sqrt{\frac{s(a+b-c)s(a+c-b)s(b+c-a)}{s(a+b+c)}}= \left\{
	\begin{array}{ll}
		r & \quad\text{(Euclidean)}\\
		\tan r & \quad\text{(spherical)}\\
		\tanh r & \quad\text{(hyperbolic)}
	\end{array}
	\right.
\end{equation}

\section{Strengthenings of Euler's Inequality}
In \cite{svrtan_wu}, Svrtan and Wu give a strengthened form of Euler's inequality in Euclidean geometry:
\begin{equation}\label{ineq:strong0}
	\frac{R}{r}\geq\frac{abc+a^3+b^3+c^3}{2abc}\geq\frac{a}{b}+\frac{b}{c}+\frac{c}{a}-1\geq\frac{2}{3}\left(\frac{a}{b}+\frac{b}{c}+\frac{c}{a}\right)\geq2.
\end{equation}
In \cite{reu2016}, Black and Smith prove that (\ref{ineq:strong0}) generalizes to spherical geometry as 
\begin{align*}
	\frac{\tan R}{\tan r}&\geq\frac{s(a)s(b)s(c)+s(a)^3+s(b)^3+s(c)^3}{2s(a)s(b)s(c)}\\&\geq\frac{s(a)}{s(b)}+\frac{s(b)}{s(c)}+\frac{s(c)}{s(a)}-1\\&\geq\frac{2}{3}\left(\frac{s(a)}{s(b)}+\frac{s(b)}{s(c)}+\frac{s(c)}{s(a)}\right)\\&\geq2.
\end{align*}
Moreover, they demonstrate that the same generalization does not work for hyperbolic geometry. They also provide the following unified strengthening of Euler's inequality for all three geometries.
\begin{equation}
	2\leq\frac{2s\left(\frac{a+b}{2}\right)s\left(\frac{b+c}{2}\right)s\left(\frac{a+c}{2}\right)}{s(a)s(b)s(c)}\leq\left\{
	\begin{array}{ll}
		R/r & \quad\text{(Euclidean)}\\
		\tan R/\tan r & \quad\text{(spherical)}\\
		\tanh R/\tanh r & \quad\text{(hyperbolic)}
	\end{array}
	\right.
\end{equation}

The strengthenings of Euler's inequality listed below are provided in \cite{adv}. In the section to follow, these inequalities are generalized to spherical or hyperbolic geometry.

\begin{gather}
	\frac{R}{2r}\geq\frac{(a+b+c)(a^3+b^3+c^3)}{(ab+bc+ca)^2}\geq 1\label{ineq:1}\\
	2R^2+r^2\geq\frac{1}{4}(a^2+b^2+c^2)\geq 3r(2R-r)\label{ineq:2}\\
	\frac{1}{4r^2}\geq\frac{1}{a^2}+\frac{1}{b^2}+\frac{1}{c^2}\geq\frac{1}{3}\left(\frac{1}{a}+\frac{1}{b}+\frac{1}{c}\right)^2\geq\frac{1}{2rR}\label{ineq:3}
\end{gather}

\section{Generalization Theorem}

\subsection{Preliminaries}

Our primary result is Theorem \ref{thm:main}, which generates a simple analogue in either spherical or hyperbolic geometry for triangle inequalities relating a function of $R$ and $r$ to a function of $a$, $b$, and $c$. For a triangle in Euclidean, spherical, or hyperbolic geometry with side lengths $a,b,c$ let us define the following quantities:
\begin{align*}
	J&:=\sqrt{s(a+b-c)s(a+c-b)s(b+c-a)},\\
	\Jbar&:=\sqrt{(s(a)+s(b)-s(c))(s(a)+s(c)-s(b))(s(b)+s(c)-s(a))}.
\end{align*}
Note that each quantity under the square root is positive by the triangle inequality. We can rewrite the formulas for $R$ and $r$ as follows.

\begin{align}
	\label{form:RJ}
	\frac{2s(a)s(b)s(c)}{J\cdot\sqrt{s(a+b+c)}} &= \left\{
	\begin{array}{ll}
		R & \quad\text{(Euclidean)}\\
		\tan R & \quad\text{(spherical)}\\
		\tanh R & \quad\text{(hyperbolic)}
	\end{array}
	\right.\\
	\label{form:rJ}
	\frac{J}{\sqrt{s(a+b+c)}} &=
	\left\{
	\begin{array}{ll}
		r & \quad\text{(Euclidean)}\\
		\tan r & \quad\text{(spherical)}\\
		\tanh r & \quad\text{(hyperbolic)}
	\end{array}
	\right.
\end{align}

The following relation between $J$ and $\Jbar$ was proved in \cite{reu2016}.

\begin{lemma}\label{lemma:JJ}
	In spherical geometry,
	\[
	J\leq\Jbar\leq\sqrt{s(a)s(b)s(c)}.
	\]
	In hyperbolic geometry,
	\[
	\Jbar\leq J\leq\sqrt{s(a)s(b)s(c)}.
	\]
	Moreover, in both geometries, $J=\Jbar$ if and only if $a=b=c$.
\end{lemma}
\begin{proof}
	First we show $J\leq\Jbar$ in spherical geometry. We omit the proof of $J\geq\Jbar$ in hyperbolic geometry, as it can be obtained simply by reversing the inequalities and replacing sine and cosine with their hyperbolic counterparts. Note that while cosine is decreasing on the interval $[0,\pi/2]$, hyperbolic cosine is increasing on $[0,\infty)$.
	
	To prove $J\leq\Jbar$ in spherical geometry, we assume $a\geq b\geq c$ without loss of generality, then verify the following two statements separately:
	\begin{gather}
		s(b+c-a)\leq s(b)+s(c)-s(a),\label{p1}\\
		s(a+b-c)s(a+c-b)\leq (s(a)+s(b)-s(c))(s(a)+s(c)-s(b))\label{p2}.
	\end{gather}
	Note that $a\geq b$ implies $2a-b-c\geq b-c$, and since cosine is decreasing on $[0,\pi/2]$ we have \begin{equation}\label{id:0}
		\cos\frac{2a-b-c}{4}\leq\cos\frac{b-c}{4}.
	\end{equation}
	We also make use of the identities
	
	\begin{equation}\label{id:1}
		\sin x+\sin y=2\sin\frac{x+y}{2}\cos\frac{x-y}{2}
	\end{equation}
	and
	\begin{equation}\label{id:2}
		\sin(x+y)\sin(x-y)=\sin^2x-\sin^2y
	\end{equation}
	(hyperbolic sine and cosine also satisfy these identities). Now we show (\ref{p1}):
	\begin{align*}
		s(b+c-a)&=\sin\frac{b+c-a}{2}\\
		&=\left(\sin\frac{a}{2}+\sin\frac{b+c-a}{2}\right)-\sin\frac{a}{2}\\
		&=2\sin\frac{b+c}{4}\cos\frac{2a-b-c}{4}-\sin\frac{a}{2}\\
		&\leq2\sin\frac{b+c}{4}\cos\frac{b-c}{4}-\sin\frac{a}{2}\\
		&=\sin\frac{b}{2}+\sin\frac{c}{2}-\sin\frac{a}{2}\\
		&=s(b)+s(c)-s(a).
	\end{align*}
	On the other hand, (\ref{p2}) is equivalent to
	\begin{align*}
		&& \sin\frac{a+b-c}{2}\sin\frac{a+c-b}{2}&\leq \left(\sin\frac{a}{2}+\sin\frac{b}{2}-\sin\frac{c}{2}\right)\left(\sin\frac{a}{2}+\sin\frac{c}{2}-\sin\frac{b}{2}\right)\\
		& \Longleftrightarrow & \sin^2\frac{a}{2}-\sin^2\frac{b-c}{2}&\leq \sin^2\frac{a}{2}-\left(\sin\frac{b}{2}-\sin\frac{c}{2}\right)^2\\
		& \Longleftrightarrow & \left(\sin\frac{b}{2}-\sin\frac{c}{2}\right)^2 &\leq \sin^2\frac{b-c}{2}\\
		& \Longleftrightarrow & \sin\frac{b}{2}-\sin\frac{c}{2} &\leq \sin\frac{b-c}{2}\\
		& \Longleftrightarrow & 2\sin\frac{b-c}{4}\cos\frac{b+c}{4} &\leq 2\sin\frac{b-c}{4}\cos\frac{b-c}{4}
	\end{align*}
	which also follows from the fact that cosine is decreasing on $[0,\pi/2]$. Moreover, observe that there is equality in (\ref{p1}) if and only if $a=b$ and equality in (\ref{p2}) if and only if $b=c$. Thus $J=\Jbar$ if and only if $a=b=c$, as desired.
	
	Next we prove $\Jbar^2\leq s(a)s(b)s(c)$ for spherical geometry. Since $s(a),s(b),s(c)$ satisfy the triangle inequality (by Lemma \ref{lemma:exist}), we can make the substitution $s(a)=x+y$, $s(b)=x+z$, $s(c)=y+z$ which transforms the inequality into
	\[
	(2x)(2y)(2z)\leq(x+y)(x+z)(y+z)
	\]
	or
	\[
	xyz\leq\frac{2xyz+x^2y+x^2z+xy^2+y^2z+xz^2+yz^2}{8}
	\]
	which is a direct application of the arithmetic mean--geometric mean inequality.
	
	Finally, we prove $J^2\leq s(a)s(b)s(c)$ for hyperbolic geometry. With the substitution $a=x+y$, $b=x+z$, $c=y+z$ this is equivalent to 
	\begin{align*}
		&& s(2x)s(2y)s(2z)&\leq s(x+y)s(x+z)s(y+z)\\
		& \Longleftrightarrow & \sinh x\sinh y\sinh z&\leq\hsf{x+y}\hsf{x+z}\hsf{y+z}\\
		& \Longleftrightarrow & 8\hsf{x}\cosh\frac{x}{2}\hsf{y}\cosh\frac{y}{2}\hsf{z}\cosh\frac{z}{2}&\leq\left(\hsf{x}\cosh\frac{y}{2}+\hsf{y}\cosh\frac{x}{2}\right)\cdots
	\end{align*}
	which also follows from the arithmetic mean--geometric mean inequality.
\end{proof}

\subsection{Theorem and Proof}

\begin{theorem}\label{thm:main}
	Let $f(x,y)$ and $g(x,y,z)$ be homogeneous functions of degree $n$. Suppose \[f(R,r)\geq g(a,b,c)\] holds for all Euclidean triangles with side lengths $a,b,c$ with equality if and only if $a=b=c$. Then:
	\begin{itemize}
		\item[(a)] if $f\left(\frac{2M^2}{x},x\right)$ is a decreasing function of $x$ for $0<x\leq M$ then
		\[
		f(\tan R,\tan r)\geq 2^n\cdot g(s(a),s(b),s(c))\cdot\left(\frac{s(a)+s(b)+s(c)}{s(a+b+c)}\right)^{n/2}
		\]
		holds for all spherical triangles with equality if and only if $a=b=c$;
		\item[(b)] if $f\left(\frac{2M^2}{x},x\right)$ is an increasing function of $x$ for $0<x\leq M$ then
		\[
		f(\tanh R,\tanh r)\geq2^n\cdot g(s(a),s(b),s(c))\cdot\left(\frac{s(a)+s(b)+s(c)}{s(a+b+c)}\right)^{n/2}
		\]
		holds for all hyperbolic triangles with equality if and only if $a=b=c$.
	\end{itemize}
\end{theorem}

\begin{proof}
	We only include the proof of part (a), as we can obtain the proof of part (b) by replacing all spherical functions with their hyperbolic counterparts.
	
	Let $T$ be a spherical triangle with side lengths $a,b,c$. By Lemma \ref{lemma:exist} there exists a Euclidean triangle $T'$ with side lengths $a'=2s(a)$, $b'=2s(b)$ and $c'=2s(c)$. By (\ref{form:R}), the circumradius of $T'$ is
	\begin{align*}
		R'
		&=\frac{2s(a)s(b)s(c)}{\sqrt{(s(a)+s(b)-s(c))(s(a)+s(c)-s(b))(s(b)+s(c)-s(a))(s(a)+s(b)+s(c))}}\\
		&=\frac{2s(a)s(b)s(c)}{\Jbar\cdot\sqrt{s(a)+s(b)+s(c)}}.
	\end{align*}
	Similarly, its inradius is 
	\[
	r'=\frac{\Jbar}{\sqrt{s(a)+s(b)+s(c)}}.
	\]
	Since $T'$ is Euclidean, we know
	\begin{equation}
		\label{i0}
		f(R',r')\geq g(a',b',c').
	\end{equation}
	Substituting the expressions for $a'$, $b'$, $c'$, $R'$, and $r'$, we have
	\[
	f\left(\frac{2s(a)s(b)s(c)}{\Jbar\cdot\sqrt{s(a)+s(b)+s(c)}},
	\frac{\Jbar}{\sqrt{s(a)+s(b)+s(c)}}\right)\geq g(2s(a),2s(b),2s(c)).
	\]
	Since $f$ and $g$ are homogeneous of degree $n$ we can write this as
	\begin{equation}\label{i2}
		f\left(\frac{2s(a)s(b)s(c)}{\Jbar},\Jbar\right)\cdot\left(\frac{1}{s(a)+s(b)+s(c)}\right)^{n/2}\geq 2^n\cdot g(s(a),s(b),s(c)).
	\end{equation}
	
	The condition that $f\left(\frac{2M^2}{x},x\right)$ is a decreasing function of $x$ on $(0,M]$ implies $f\left(\frac{2M^2}{x_1},x_1\right)\geq f\left(\frac{2M^2}{x_2},x_2\right)$ for all $x_1,x_2$ satisfying $0<x_1\leq x_2\leq M$. By Lemma \ref{lemma:JJ}, we have $0<J\leq \Jbar\leq\sqrt{s(a)s(b)s(c)}$ for spherical triangles. Taking $x_1=J$, $x_2=\Jbar$ and $M=\sqrt{s(a)s(b)s(c)}$ yields
	\begin{equation}\label{i1}
		f\left(\frac{2s(a)s(b)s(c)}{J},J\right)\geq f\left(\frac{2s(a)s(b)s(c)}{\Jbar},\Jbar\right).
	\end{equation}
	
	Combining (\ref{i2}) and (\ref{i1}), we have
	\[
	f\left(\frac{2s(a)s(b)s(c)}{J},J\right)\cdot\left(\frac{1}{s(a)+s(b)+s(c)}\right)^{n/2}\geq 2^n\cdot g(s(a),s(b),s(c)).
	\]
	Now we multiply both sides by the quantity $\left(\frac{s(a)+s(b)+s(c)}{s(a+b+c)}\right)^{n/2}$ to get the desired right hand side of
	\[
	2^n\cdot g(s(a),s(b),s(c))\cdot \left(\frac{s(a)+s(b)+s(c)}{s(a+b+c)}\right)^{n/2}
	\]
	while the left hand side becomes
	\[
	f\left(\frac{2s(a)s(b)s(c)}{J},J\right)\cdot\left(\frac{1}{s(a+b+c)}\right)^{n/2}.
	\]
	Since $f$ is homogeneous we can rewrite the left hand side:
	\begin{align*}
		f\left(\frac{2s(a)s(b)s(c)}{J},J\right)\cdot\left(\frac{1}{s(a+b+c)}\right)^{n/2}&=f\left(\frac{2s(a)s(b)s(c)}{J\cdot \sqrt{s(a+b+c)}},\frac{J}{\sqrt{s(a+b+c)}}\right)\\
		&=f(\tan R,\tan r)
	\end{align*}
	by formulas (\ref{form:RJ}) and (\ref{form:rJ}). So, we have
	\[
	f(\tan R,\tan r)\geq2^n\cdot g(s(a),s(b),s(c))\cdot \left(\frac{s(a)+s(b)+s(c)}{s(a+b+c)}\right)^{n/2}
	\]
	as desired.
	
	
	We only need to show that equality is achieved if and only if $a=b=c$. Note that equality is achieved if and only if there is equality in both (\ref{i2}) and (\ref{i1}). Since (\ref{i2}) is equivalent to (\ref{i0}), there is equality if and only if $a'=b'=c'$ which is true if and only if $a=b=c$. On the other hand, there is equality in (\ref{i1}) if and only if $J=\Jbar$, which occurs if and only if $a=b=c$ by Lemma \ref{lemma:JJ}.
	
\end{proof}

The following corollary immediately generalizes (\ref{ineq:strong0}) and (\ref{ineq:1}) to spherical geometry.

\begin{corollary}\label{thm:s2Rr}
	
	Let $\frac{R}{r}\geq g(a,b,c)$ be an inequality which holds for all Euclidean triangles with circumradius $R$, inradius $r$, and side lengths $a,b,c$. Then for all spherical triangles,
	\[
	\frac{\tan R}{\tan r}\geq g(s(a),s(b),s(c)).
	\]
	
\end{corollary}

\begin{proof}
	Let $f(x,y)=x/y$. Then we have
	\[
	f(R,r)\geq g(a,b,c)
	\]
	with $f$ and $g$ homogeneous functions of degree $n=0$. Now since
	\[
	f\left(\frac{2M^2}{x},x\right)=\frac{2M^2}{x^2}
	\]
	is a decreasing function of $x$ for $x>0$, part (a) of Theorem \ref{thm:main} yields
	\[
	f(\tan R,\tan r)\geq g(s(a),s(b),s(c))
	\]
	for all spherical triangles, as desired.
\end{proof}

\subsection{Examples}

To further illustrate the use of Theorem \ref{thm:main} we apply it to (\ref{ineq:2}) and (\ref{ineq:3}). These two statements give rise to four separate inequalities relating $R$ and $r$ to the side lengths $a,b,c$ for which we can find an analogue in either spherical or hyperbolic geometry.

\begin{proposition}[Generalization of (\ref{ineq:2}) (left inequality) to spherical geometry]
	All spherical triangles satisfy
	\begin{equation}
		2\tan^2 R+\tan^2 r\geq \left(s(a)^2+s(b)^2+s(c)^2\right)\left(\frac{s(a)+s(b)+s(c)}{s(a+b+c)}\right)
	\end{equation}
	with equality if and only if $a=b=c$.
\end{proposition}

\begin{proof}
	The leftmost inequality in (\ref{ineq:2}) states that all Euclidean triangles satisfy
	\begin{equation}
		\label{ineq:2a}
		2R^2+r^2\geq\frac{1}{4}(a^2+b^2+c^2).
	\end{equation}
	Let $f(x,y)=2x^2+y^2$ and $g(x,y,z)=\frac{1}{4}(x^2+y^2+z^2)$. Note that $f$ and $g$ are homogeneous of degree 2, and (\ref{ineq:2a}) can be stated as
	\[
	f(R,r)\geq g(a,b,c).
	\]
	To prove the desired spherical analogue, we only need to confirm that
	\[
	f\left(\frac{2M^2}{x},x\right)=\frac{8M^4}{x^2}+x^2
	\]
	is a decreasing function of $x$ for $x\in(0,M]$. We take its derivative with respect to $x$:
	\begin{align*}
		\frac{d}{dx}\left[\frac{8M^4}{x^2}+x^2\right]&=-\frac{16M^4}{x^3}+2x\\
		&=\frac{2x^4-16M^4}{x^3}
	\end{align*}
	which is indeed negative for $0<x\leq M$. Theorem $\ref{thm:main}$ tells us that for spherical triangles,
	\[
	f(\tan R,\tan r)\geq 2^2\cdot g(s(a),s(b),s(c))\cdot\left(\frac{s(a)+s(b)+s(c)}{s(a+b+c)}\right)^{2/2}
	\]
	or
	\[
	2\tan^2 R+\tan^2 r\geq \left(s(a)^2+s(b)^2+s(c)^2\right)\left(\frac{s(a)+s(b)+s(c)}{s(a+b+c)}\right).
	\]
\end{proof}

\begin{proposition}[Generalization of (\ref{ineq:2}) (right inequality) to hyperbolic geometry.]
	All hyperbolic triangles satisfy
	\begin{equation}\label{ineq:2b}
		\left(s(a)^2+s(b)^2+s(c)^2\right)\left(\frac{s(a)+s(b)+s(c)}{s(a+b+c)}\right)
		\geq 3\tanh r(2\tanh R-\tanh r)
	\end{equation}
	with equality if and only if $a=b=c$.
\end{proposition}

\begin{proof}
	The rightmost inequality in (\ref{ineq:2}) states that all Euclidean triangles satisfy
	\[
	\frac{1}{4}(a^2+b^2+c^2)\geq 3r(2R-r).
	\]
	Let $f(x,y)=-3y(2x-y)$ and $g(x,y,z)=-\frac{1}{4}(x^2+y^2+z^2)$. Then $f$ and $g$ are homogeneous of degree 2 and
	\[
	f(R,r)\geq g(a,b,c)
	\]
	holds for all Euclidean triangles. By part (b) of Theorem \ref{thm:main}, since
	\[
	f\left(\frac{2M^2}{x},x\right)=-3x\left(\frac{4M^2}{x}-x\right)=3x^2-12M^2
	\]
	is an increasing function of $x$ for all $x>0$, the inequality
	\[
	f(\tanh R,\tanh r)\geq 2^2\cdot g(s(a),s(b),s(c))\cdot\left(\frac{s(a)+s(b)+s(c)}{s(a+b+c)}\right)^{2/2}
	\]
	which is equivalent to
	\[
	\left(s(a)^2+s(b)^2+s(c)^2\right)\left(\frac{s(a)+s(b)+s(c)}{s(a+b+c)}\right)
	\geq 3\tanh r(2\tanh R-\tanh r)
	\]
	holds for all hyperbolic triangles.
\end{proof}

Theorem \ref{thm:main} also proves the following two propositions, generalizing (\ref{ineq:3}).

\begin{proposition}[Generalization of (\ref{ineq:3}) (left inequality) to spherical geometry]\label{ineq:3a}
	All spherical triangles satisfy
	\begin{equation}
		\frac{1}{\tan^2r}\geq\left(\frac{1}{s(a)^2}+\frac{1}{s(b)^2}+\frac{1}{s(c)^2}\right)\left(\frac{s(a)+s(b)+s(c)}{s(a+b+c)}\right)^{-1}
	\end{equation}
	with equality if and only if $a=b=c$.
\end{proposition}

\begin{proposition}[Generalization of (\ref{ineq:3}) (right inequality) to spherical and hyperbolic geometry]\label{ineq:3b}
	\leavevmode
	
	\begin{itemize}
		\item[(a)] All spherical triangles satisfy
		\begin{equation}
			\frac{1}{3}\left(\frac{1}{s(a)}+\frac{1}{s(b)}+\frac{1}{s(c)}\right)^2\left(\frac{s(a)+s(b)+s(c)}{s(a+b+c)}\right)^{-1}\geq\frac{2}{\tan r\tan R}
		\end{equation}
		with equality if and only if $a=b=c$.
		
		\item[(b)] All hyperbolic triangles satisfy
		\begin{equation}
			\frac{1}{3}\left(\frac{1}{s(a)}+\frac{1}{s(b)}+\frac{1}{s(c)}\right)^2\left(\frac{s(a)+s(b)+s(c)}{s(a+b+c)}\right)^{-1}\geq\frac{2}{\tanh r\tanh R}
		\end{equation}
		with equality if and only if $a=b=c$.
	\end{itemize}
\end{proposition}

\section{Euler's Inequality in $n$-Dimensional Space}

We now turn to higher-dimensional space, namely $n$-dimensional spherical space, modeled by the unit sphere in $\mathbb{R}^{n+1}$. We focus on the generalization of the known inequality
\begin{equation}\label{eq:Rnr}
	R\geq nr
\end{equation}
for an $n$-dimensional Euclidean simplex with circumradius $R$ and inradius $r$.


\begin{theorem}[Extension of Corollary \ref{thm:s2Rr}]\label{thm:snRr}
	
	Let $\frac{R}{r}\geq f(\{d_{ij}\})$ be an inequality which holds for an $n$-dimensional Euclidean simplex (where $R$ is the circumradius, $r$ is the inradius, and $\{d_{ij}\}$ is the set of edge lengths). Then $\frac{\tan R}{\tan r}\geq f(\{s(d_{ij})\})$ holds for an $n$-dimensional spherical simplex.
	
\end{theorem}

\begin{figure}
	\begin{center}
		\includegraphics[width=0.35\columnwidth]{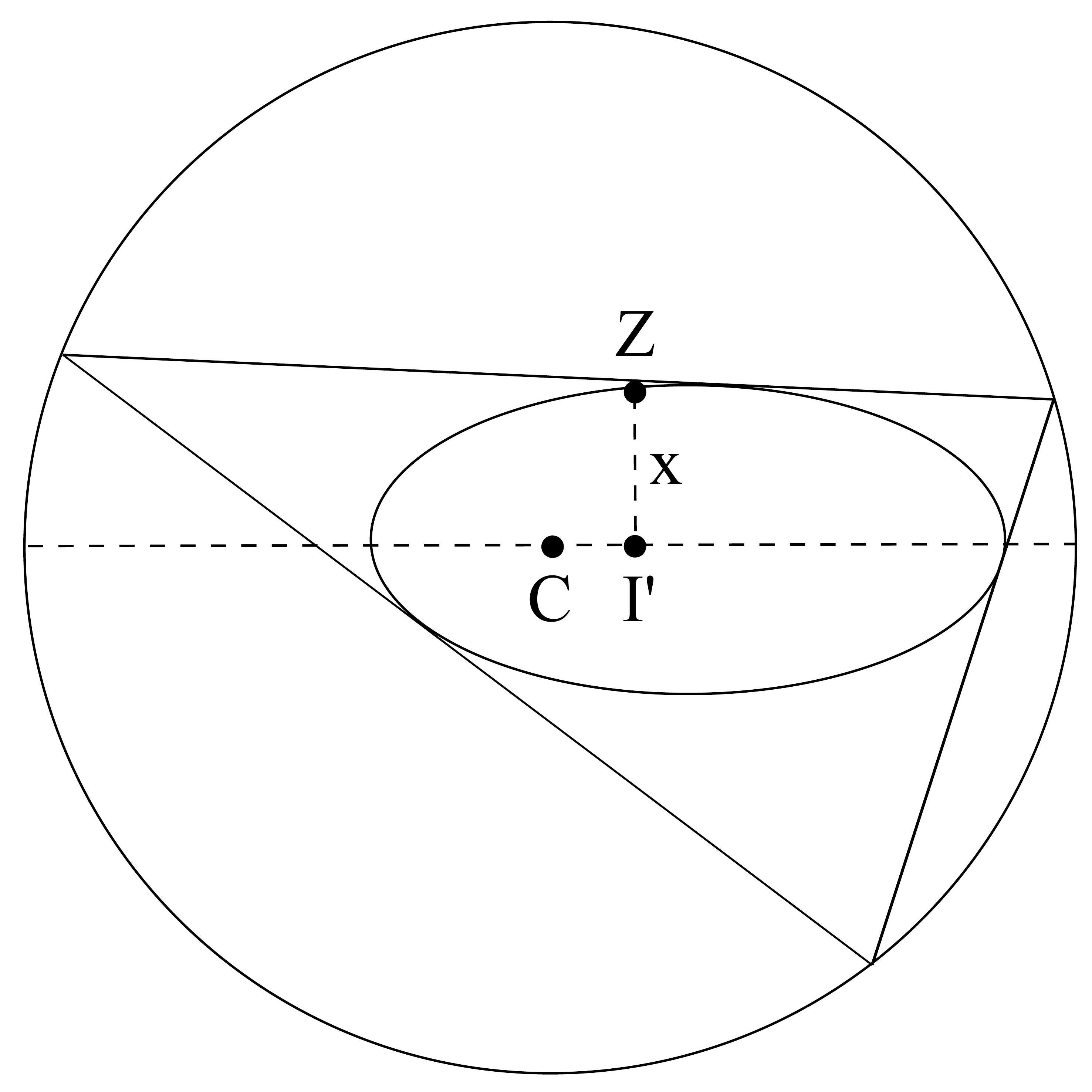}
		\caption{The projection of the simplex, circumsphere, and insphere onto the tangent plane $x_{n+1}=1$.}
		\label{fig:inellipse}
	\end{center}
\end{figure}

\begin{proof}
	Assume without loss of generality that a spherical $n$-simplex, lying on the unit sphere in $\mathbb{R}^{n+1}$, has its circumcenter at the north pole $C=(0,\dots,0,1)$. We may also assume the incenter lies on the two-dimensional $x_1x_{n+1}$ plane. By projecting the simplex onto the $n$-dimensional hyperplane tangent to the unit sphere at $C$ (with equation $x_{n+1}=1$) we obtain an $n$-dimensional Euclidean simplex with circumradius $\tan R$ and some inradius $r'$. 
	On the other hand, the image of its insphere under the same projection is an inscribed ellipsoid of the Euclidean simplex. By symmetry, this ellipsoid has the same radius $b$ in all directions except along the $x_1$-axis where it is longer. If $d_{ij}$ denotes the length of the spherical edge between vertices $P_i$ and $P_j$ then the corresponding Euclidean edge length is $k\cdot\sin\frac{d_{ij}}{2}=k\cdot s(d_{ij})$ where $k=2/\cos R$. By similarity, there exists a Euclidean simplex with side lengths $\{s(d_{ij})\}$ and the same circumradius-to-inradius ratio. Since the inequality is true for a Euclidean simplex we have
	\[
	\frac{\tan R}{r'}\geq f(\{s_s(d_{ij})\})
	\]
	and so it is sufficient to show
	\[
	\frac{\tan R}{\tan r}\geq\frac{\tan R}{r'}.
	\]
	or
	\[
	r'\geq\tan r.
	\]
	
	Let $O$ be the origin, $I$ the image of the spherical incenter and $Z$ the point where the ray starting at $I$ and heading in the $x_2$-direction intersects the inscribed ellipse (see Fig. \ref{fig:inellipse}). Observe that $OZI$ is a right triangle yielding $\tan r=\frac{IZ}{IO}$. It is clear that $IO\geq 1$ (as the $x_{n+1}$-coordinate of $I$ is equal to 1) so we have $\tan r\leq IZ$. But $IZ\leq b$, and $b\leq r'$ as a sphere of radius $b$ with the same center as the inscribed ellipse would be contained in the ellipse and also the simplex. Thus $\tan r\leq r'$.
\end{proof}

Theorem \ref{thm:snRr} immediately generalizes (\ref{eq:Rnr}) to spherical geometry.

\begin{theorem}
	
	An $n$-dimensional simplex in spherical geometry with circumradius $R$ and inradius $r$ satisfies
	\[
	\tan R\geq n\tan r.
	\]
	
\end{theorem}

\section*{Acknowledgement}

This work was done during the 2017 REU program in mathematics at Oregon State University, with support by National Science Foundation Grant DMS-1359173, under the supervision of Professor Ren Guo.

\vskip1cm

\noindent\textit{Karina Cho\\
	Harvey Mudd College}\\
\texttt{karinaecho@gmail.com}

\vskip0.5cm

\noindent\textit{Jacob Naranjo\\
	Kalamazoo College}\\
\texttt{jaconaranj57@gmail.com}


\newpage
\begin{thebibliography}{99}

\bibitem{reu2016}Estonia Black and Caleb Smith.
\textit{Strengthened Euler's Inequality in Spherical and Hyperbolic Geometries.} REU at Oregon State University, 2016.

\bibitem{guo}Ren Guo and Nilg\"{u}n S\"{o}nmez.
\textit{Cyclic Polygons in Classical Geometry.}
C.R. Acad. Bulgare Sci., 64(2), September 2010.

\bibitem{adv}D.S. Mitrinovi\'{c}, J.E. Pe\v{c}ari\'{c}, and V. Volenec.
\textit{Recent Advances in Geometric Inequalities.}
Kluwer Academic Publishers, 1989.

\bibitem{svrtan_veljan} Dragutin Svrtan and Darko Veljan.
\textit{Non-Euclidean Versions of Some Classical Triangle Inequalities.}
Forum Geometricorum, Vol. 12: 197-209, 2012.

\bibitem{svrtan_wu} Dragutin Svrtan and Shanhe Wu.
\textit{Parametrized Klamkin's inequality and improved Euler's inequality.}
Math. Inequal. Appl., 11(4): 729-737, 2008.

\end{thebibliography}
\end{document}